\documentclass[12pt,reqno]{amsart}

\usepackage{AKstyle}

\begin{document}

\author{Jean Bourgain}
\thanks{Bourgain is partially supported by NSF grant DMS-0808042.}
\email{bourgain@ias.edu}
\address{IAS, Princeton, NJ}
\author{Alex Kontorovich}
\thanks{Kontorovich is partially supported by  NSF grants DMS-0802998 and DMS-0635607, and the Ellentuck Fund at IAS}
\email{avk@math.ias.edu}
\address{IAS and Brown, Princeton, NJ}

\title[Integers in thin subgroups of $\SL_{2}(\Z)$]{On representations of integers in thin subgroups of $\SL_{2}(\Z)$}

\begin{abstract}
Let $\G<\SL(2,\Z)$ be a free, finitely generated Fuchsian group of the second kind with no parabolics, and fix  two primitive vectors  $v_{0},w_{0}\in\Z^{2}\setminus\{0\}$. We consider the set $\cS$ of all integers occurring in $v_{0}\g\,{}^{t}w_{0}$, for $\g\in\G$. Assume that the limit set of $\G$ has Hausdorff dimension $\gd>0.99995$, that is, $\G$ is thin but not too thin. Using a variant of the circle method, new bilinear forms estimates and Gamburd's $5/6$-th spectral gap in infinite-volume, we show that $\cS$ contains almost all of its admissible primes, that is, those not excluded by local (congruence) obstructions. Moreover, we show that the exceptional  set $\fE(N)$ of integers $|n|<N$ which are locally admissible ($n\in\cS(\mod q)$ for all $q\ge1$) but fail to be globally represented, $n\notin \cS$, has a power savings, $|\fE(N)|\ll N^{1-\vep_{0}}$ for some $\vep_{0}>0$.
\end{abstract}
\date{\today}
\maketitle
\tableofcontents

\section{Introduction}

Recently Bougain, Gamburd and Sarnak \cite
{BourgainGamburdSarnak2006,BourgainGamburdSarnak2008}   introduced the
Affine Linear Sieve, which
concerns the
 application of various sieve methods to the setting of (possibly thin) orbits of groups 
of morphisms of affine $n$-space. 
 Until now, the Affine Linear Sieve
 had produced 
almost-primes
 in great generality, in some cases giving explicit bounds for the number of factors 
(e.g. \cite{Kontorovich2009, KontorovichOh2009}), but had not yet exhibited actual 
primes in thin orbits. The failure of sieve  methods to produce primes stems from 
the well-known {\it parity barrier}, pinpointed by Selberg 60 years ago, that sieves 
alone cannot distinguish between integers having an odd or even number of prime 
factors. 
In the 1930s, I. M. Vinogradov introduced bilinear forms estimates to overcome this 
barrier, leading to his resolution of the ternary Goldbach problem.
It is our present goal to  inject bilinear forms methods into the Affine Linear Sieve to
 produce 
 {\it primes} and not just almost primes in sets coming from thin orbits.

\subsection{Statement of the Main Theorem}\

The (multi-) set $\cS$ of integers which we study is the following. Fix two primitive 
vectors $v_{0},w_{0}\in\Z^{2}$, and let $\G<\SL(2,\Z)$ 
be 
finitely-generated, free, contain no parabolic elements, and whose limit set has  
Hausdorff dimension $\gd>0.99995$ (so $\G$ is thin but not too thin). 
Such $\G$ exist, cf. Remark \ref{rmk:GamEG}.
Let 
\be\label{cSis}
\cS=\<v_{0}\cdot\G,w_{0}\>
=
\bigg\{
\<v_{0}\cdot\g,w_{0}\>
:
\g\in\G
\bigg\}
,
\ee
where the inner product is the usual one on $\R^{2}$, and the $\G$-action is usual matrix 
multiplication. 
\\

For ease of exposition, we focus on the example   $v_{0}=w_{0}=(0,1)$, for which
$
\<v_{0}\g,w_{0}\>
=
d_{\g}
,
$
where
$
\g=
\bigl( \begin{smallmatrix}
a_{\g}&b_{\g}\\ c_{\g}&d_{\g}
\end{smallmatrix} \bigr)
.
$
So in this case,
$\cS$ consists of all lower right entries $d_{\g}$ for $\g\in\G$.
The goal, then, is to prove that given any  $\G<\SL(2,\Z)$ as above, as long as it is 
not too thin (as measured by $\gd$), the set of lower right entries contains infinitely 
many prime values. (A pleasant feature of $d_{\g}$ is that there is always an infinite set of admissible primes, as we see below.)\\

In fact, we  show more: the set of integers appearing in $\cS$ has full density, with a 
power savings in the exceptional set! By this we mean the following. 

\begin{thm}\label{thm:main}
Let $\G<\SL(2,\Z)$ be finitely-generated, free, and have no parabolics, and let $\cS$ be as in \eqref{cSis}. Let $\gd$ be the Hausdorff dimension  of the limit set of 
$\G$. Assume that $\gd$ is bounded below by the largest root of the polynomial 
$$
1020 - 8897 x - 5010 x^2 + 12888 x^3,
$$ 
that is,
$$
\gd>0.9999493550.
$$
Let $\frak E(N)$ be the set of integers $|n|<N$ which fail the local-to-global 
principle, that is,  $n$ is admissible
($n\in\cS(\mod q)$ for all integers $q\ge1$),  but nevertheless $n\notin 
\cS$.
Then there is some $\vep_{0}>0$ such that 
\be\label{eq:main}
|\frak E(N)|\ll N^{1-\vep_{0}}
,
\ee
as $N\to \infty$.
\end{thm}

\begin{rmk}
By the prime number theorem, 
the 
exceptional set $\frak E(N)$ cannot contain even a positive proportion of the 
admissible primes, because then it would have size $\gg {N\over \log N}$, 
contradicting the power savings in \eqref{eq:main}. For the choice $v_{0}=w_{0}=(0,1)$ giving 
$d_{\g}$, there are always  admissible residue classes of primes, since every group contains the identity element 
$\g=I$ with $d_{\g}=1$.
\end{rmk}

\begin{rmk}
We have made no attempt at optimizing the allowed range of $\gd$; this can surely 
be done with some more effort, but our proof is sufficiently involved as is. Our main 
point is that {\it some} $\gd<1$ is allowed, and that  this range can be made explicit. 
\end{rmk}

\begin{rmk}\label{MontVaugh}
One should note the parallel between \eqref{eq:main} and the theorem of 
Montgomery and Vaughan \cite{MontgomeryVaughan1975}, that the exceptional set in the Goldbach  problem, that is the set $\fE(X)$ of even integers at most $X$ which cannot be expressed as the sum of two primes, has a power savings, $|\fE(X)|\ll X^{1-\vep_{0}}$. (Of course the full Goldbach problem is equivalent to $|\fE(X)|=1$ for $X>2$.)  A big difference between the two is that Goldbach is a definite problem, in that there are only finitely many chances to find primes $p_{1},p_{2}$ with $p_{1}+p_{2}=2n$, whereas our present problem is indefinite: one can take larger and larger balls in $\G$, whose individual entries may attain the sought-after values; see \eqref{cMFake}.
\end{rmk}

\begin{rmk}\label{rmk:GamEG}
There exist groups $\G$ satisfying the above conditions. 
For just one example, recall 
the commutator group $\G'(2)=[\G(2),\G(2)]$ of the classical congruence subgroup $\G(2)$ of level $2$ in $\PSL(2,\Z)$. That is, let $A=\mattwo1201$ and $B=\mattwo1021$; then $\G(2)=\<A,B\>$ is free, and $\G'(2)$ consists of all elements of the form
$$
A^{n_{1}}B^{m_{1}}
A^{n_{2}}B^{m_{2}}
\cdots
A^{n_{k}}B^{m_{k}}
,
$$
with $\sum_{j}n_{j}=\sum_{j}m_{j}=0$.
It is easy to see that the only trace $\pm2$ element of $\G'(2)$ is the identity. The group $\G'(2)$ is thin, having infinite index in $\SL(2,\Z)$, 
but is infinitely-generated, 
and its limit set is the entire boundary. Hence its Hausdorff dimension is  $\gd_{\G'(2)}=1$. Now, take a  subgroup $\G$ of $\G'(2)$ which is generated by finitely many elements of $\G'(2)$; it will still be free and have no parabolics. One can \cite{Sullivan1984} add more and more generators to $\G$ in such a way that  the Hausdorff dimension $\gd_{\G}$ of the limit set of $\G$ will be arbitrarily close to $1$.
\end{rmk}

\subsection{Methods}\

Our starting point is
 the Hardy-Littlewood circle method; we now 
describe the main ingredients. 
One forms an exponential sum, which is essentially
\be\label{fakeSN}
S_{N}(\gt)\approx
\sum_{\g\in\G\atop\|\g\|<N} e(\<v_{0}\g,w_{0}\>\gt)
,
\qquad\text{ where $\gt\in[0,1]$ and 
$e(x)=e^{2\pi i x}
$.}
\ee
The norm is the usual matrix norm 
$$
\left\|
\bigl( \begin{smallmatrix}
a&b\\ c&d
\end{smallmatrix} \bigr)
\right\|=\sqrt{a^{2}+b^{2}+c^{2}+d^{2}}
.
$$
Let 
\be\label{fakeRN}
R_{N}(n):=
\widehat{ S_{N}}(-n)
=
 \int_{0}^{1}S_{N}(\gt)e(-n\gt)d\gt
\approx
\sum_{\g\in\G\atop\|\g\|<N} \bo_{\{d_{\g}=n\}}
\ee
be roughly the number of representations of $n$, that is, the number of $\|\g\|<N$ having $d_{\g}=n$. Hence $n$ appears in $\cS
$  
if 
$R_{N}(n)>0$.

Divide the circle into major and minor arcs,
$[0,1] = \fM\sqcup\fm$,
 where the major arcs $\fM$ consist of a union of  small intervals
 near rationals with small denominators, 
and 
the minor arcs $\fm$ are the rest. Write
$$
\cM_{N}(n)=\int_{\fM}S_{N}(\gt)e(-n\gt)d\gt,
$$
and
$$
\cE_{N}(n)=\int_{\fm}S_{N}(\gt)e(-n\gt)d\gt,
$$
so that 
$$
R_{N}(n)=\cM_{N}(n)+\cE_{N}(n),
$$ 
the first term being the ``main'' term 
and the latter being the ``error.''

Of course 
$$
R_{N}(n)\ge \cM_{N}(n)-|\cE_{N}(n)|,
$$ 
so $n$ is representable if 
\vskip-.275in
$$
\cM_
{N}(n)>|\cE_{N}(n)|
.
$$ 
Note that by Lax-Phillips \cite{LaxPhillips1982}, the total mass is 
$$
S_{N}(0)=\sum_{|n|<N}R_{N}(n)\sim c\cdot N^{2\gd},
$$ 
and so on average, one expects numbers to appear with multiplicity roughly
\vskip-.275in
$$
N^{2
\gd-1}.
$$ 
In a more-or-less straightforward calculation using the spectral gap, we  show, 
cf.~Theorem \ref{thm:major}, that for ``almost'' all $|n|<N$, the main term is 
\be\label{cMis}
\cM_{N}(n) \gg {1\over \log\log (10+|n|)} N^{2\gd-1},
\ee
if $n$ is admissible (and  of course zero otherwise).

\subsection{Local-global?}\

One may  {\it a priori} wonder whether in fact  $\cE_{N}$ can be controlled in $\cL^
{1}$, that is, uniformly in $n$:
\be\label{falseConj}
|\cE_{N}(n)|
\le
 \int_{\fm}|S_{N}(\gt)|d\gt \overset?= o(N^{2\gd-1}).
\ee
This (and removing ``almost'' before \eqref{cMis}) would lead to a genuine local-global principle -- there would only be 
no
exceptions at all (by taking $N$ to infinity with $n$ fixed). 

But algebra intervenes, and \eqref{falseConj} is {\bf false}! Consider changing $w_
{0}$ from $(0,1)$ to $(1,0)$, which now picks up the lower left entry $c_{\g}$ instead of $d_{\g}$. 
The above analysis all goes through, and setting $n=0$, there are no local obstructions, so the main term 
is
\be\label{cMFake}
\cM_{N}(0)\gg N^{2\gd-1}.
\ee
This count is a manifestation of the fact that, as our setup is indefinite (cf. Remark \ref{MontVaugh}), there are more and more chances to hit $n=0$ by letting $N$ grow.
Hence $R_{N}(0)$ now counts the set of matrices $\g\in\G$ with $c_{\g}=0$, which are necessarily parabolic. But we
assumed that $\G$ has no parabolics (other than $I$ and possibly $-I$)! So $R_{N}(0)
\le 2$, and we have
$$
N^{2\gd-1}
\ll
|
R_{N}(0)-\cM_{N}(0)|
=
|\cE_{N}(0)|
=
\left|
\int_{\fm}S_{N}(\gt)d\gt
\right|
\le
\int_{\fm}|S_{N}(\gt)|d\gt
,
$$
thereby disproving \eqref{falseConj}. Hence the exceptional set need not be empty. 

\begin{rmk}
It may still be the case that the exceptional set  is finite. One example of such is the question of curvatures in an integral Apollonian gasket, cf. \cite{GrahamEtAl2003, SarnakToLagarias, KontorovichOh2008, BourgainFuchs2010}, where it is expected 
that every sufficiently large admissible number appears.
\end{rmk}
\begin{rmk}
One faces a similar obstacle upon trying to use the circle method to count sums of 
four squares -- one cannot introduce minor arcs because the $\cL^{1}$ norm is as 
big as the main term. This issue was overcome by Kloosterman \cite{Kloosterman1927}, but his refinement 
is impossible in our context (as again, it would imply a complete local-to-global phenomenon, which is false). 
\end{rmk}

Instead we average over $n$
, which bypasses the $\cL^{1}$ norm and places 
the $\cL^{2}$ norm in the spotlight. We prove in Theorem \ref{thm:minor}
 that  for some $\eta>0$,
\be\label{cEis}
\sum_{|n| < N}|\cE_{N}(n)|^{2}
=
\int_{\fm
}
|
S_{N}(\gt)|^{2}
d\gt
\ll 
N^{4\gd-1-\eta}
.
\ee
A straightforward argument  gives  \eqref{eq:main} from \eqref{cMis} and \eqref
{cEis}, cf. Theorem \ref{thm:finish}.

\subsection{Ingredients}\

It is the bilinear (in fact multilinear) structure of the set $\cS=\<v_{0}\G,w_{0}\>$ 
which we 
exploit. A key ingredient in our approach 
is the observation that, instead of \eqref{fakeSN},
one can consider the exponential sum $S_{N}(\gt)$ of the form
\be\label{SNapprox}
\sum_{\g_{1}\in\G\atop\|\g_{1}\|<N^{1/2}}
\sum_{\g_{2}\in\G\atop\|\g_{2}\|<N^{1/2}}
 e(\<v_{0}\g_{1}\g_{2},w_{0}\>\gt)
,
\ee
say,
whose  transform $R_{N}(n)=\widehat{S_{N}}(-n)$
is  
neither a lower nor upper bound for the true number of 
representatives, but certainly still has the property that if $R_{N}(n)>0$, then $n$ is 
representable. (Even \eqref{SNapprox} is an oversimplification of our actual 
exponential sum, for which see \S\ref{secSetup}.)

The advantage of the above formulation is that, since 
$\<v_{0}\g_{1}\g_{2},w_{0}\> = 
 \<v_{0}\g_{1},w_{0}\phantom{}^{t}\g_{2}\>$,
 one can rewrite $S_{N}(\gt)$ as 
 $$
 \sum_{x\in\Z^{2}}\sum_{y\in\Z^{2}}\mu(x)\nu(y) e(\<x,y\>\gt)
 ,
 $$ 
 for appropriate measures $\mu$ and $\nu$.
We then develop several new bounds for bilinear forms of the above type, cf.~ 
Theorems \ref{thm:min1}, \ref{thm:min2}, and \ref{thm:min3}. 

Another key ingredient in the above technology  is the ability to count effectively 
(with power savings error terms, uniformly over congruence subgroups and their 
cosets) 
the number of elements of a 
group of isometries of an
infinite-volume hyperbolic manifold
lying  
in certain restricted 
regions. 
This analysis is carried out together with Peter Sarnak in the companion paper \cite
{BourgainKontorovichSarnak2009}. 
It is our pleasure to thank him for many illuminating conversations during this work.

\subsection{Organization}\

The paper is organized as follows. In \S\ref{secSectors}, we 
state
several 
estimates on the number of orbital points in various restricted 
regions, 
which will be necessary in the sequel. Their proofs are given in \cite
{BourgainKontorovichSarnak2009}. 
In \S\ref{secSetup}, 
 $S_{N}$ is properly defined and the major and minor arcs are introduced 
(even these require a slight deviation from 
the traditional 
construction). The major arcs are 
controlled in \S\ref{secMajor}, the minor arcs are disposed of in Sections \ref
{secMin1}, \ref{secMin2}, and \ref{secMin3}, and all of the ingredients are 
assembled in the final \S\ref{secFinish}.\\

\

\

\



\section{Estimates  of Orbital  Regions}\label{secSectors}

\

In this section, we state some estimates which are used in the sequel. 
Their proofs appear in \cite
{BourgainKontorovichSarnak2009}. 

Let $\G<\SL(2,\Z)$ be a free, finitely generated Fuchsian group of the second kind with no parabolics, and let $\gd$ be the Hausdorff dimension of the limit set of $\G$. Assume throughout that $\gd>5/6$. Let $0<\gs<1/4$ be another parameter to be chosen later, cf. \eqref{gsBnd1}.
Let $q\ge1$, and define a ``congruence'' subgroup of $\G$ of level $q$ to be a group  which contains the principal congruence subgroup
$$
\G(q):=
\{\g\in\G:\g\equiv I(\mod q)\}.
$$

We require the following Sobolev-type norm. Fix $T\ge1$ and 
let $\{X_{1},X_{2},X_{3}\}$ be a basis for the Lie algebra $\fg$. 
Then
define the  $\cS_{\infty,T}$ norm by
$$
\cS_{\infty,T}f=\max_{X\in\{0,X_{1},X_{2},X_{3}\}}\sup_{g\in G, \|g\|<T}|d\pi(X).f(g)|
,
$$
that is, the supremal value of first order derivatives of $f$ in a ball of radius $T$ in $G$.\\

The following is a form of ``dualization'' (cf. \cite{NevoSarnak2009}), an analogue of Poisson summation.\\

\begin{thm}\label{thm:PassToQ}
Fix any $\g_{0}\in\G$ and a congruence subgroup $\G_{1}(q)<\G$ of level $q\ge1$. Let  $f:G\to\C$ be a smooth function with $|f|\le1$
 There is a fixed ``bad'' integer $\fB$ which depends only on $\G$ such that for  $q=q'q''$, $q'\mid\fB$, 
$$
\sum_{\g\in \g_{0}\cdot\G(q)\atop\|\g\|<T} f(\g)
=
{1\over [\G:\G(q)]}
\left(
\sum_{\g\in \G\atop\|\g\|<T} f(\g)
+\cE_{q'}
\right)
+
O\bigg(
T^{\frac67 2\gd+ \frac 5{21}}(1+ \cS_{\infty,T}f)^{6/7}
\bigg)
.
$$
Here $\cE_{q'}\ll T^{2\gd-\ga_{0}}$, with $\ga_{0}>0$, and all implied constants are independent of $q''$ and $\g_{0}$.
\end{thm}

\pf
This is Theorem 1.13 in \cite{BourgainKontorovichSarnak2009}, using Gamburd's \cite{Gamburd2002} spectral gap $\gT=5/6$.
\epf

\

\begin{thm}\label{thm:LowBnd}
Let $v_{0},w\in\Z^{2}$ and
assume that 
$n\in\Z$, $\frac N{K_{0}}<|n|<N$, $|w|< {N^{1-\gs}}$, $|v_{0}|\le 1$,
and
$
|n|<|v_{0}||w|N^{\gs}.
$
Then
$$
\sum_{\g\in\G\atop\|\g\|<N^{\gs}}
\bo{\left\{
|\<v_{0}\g,w
\>-n|<{N\over 2K_{0}}
\right\}}
\gg
{N^{2\gd\gs}\over K_{0}}
+ O\left(N^{\gs(\frac 34 + \frac 14 2\gd)}(\log N)^{1/4}\right)
.
$$
\end{thm}

\pf
This is an application of Theorem 1.14 in \cite{BourgainKontorovichSarnak2009}.
\epf

\

\begin{thm}\label{thm:UpBnd}
Fix $(c,d)$ and $y=(y_{1},y_{2})$ in $\Z^{2}$ with 
  $|y|< N^{1/2}$, $|(c,d)|<N^{\gs}$ and 
 $|y|<N^{1/2-\gs}|(c,d)|$. 
 Then
\beann
\sum_{\g\in\G\atop \|\g\|<N^{1/2-\gs}}
\bo
\left\{|(c,d)\g-y|<\frac {N^{1/2}}K\right\}
\bo
\bigg\{(c,d)\g\equiv y(\mod q) \bigg\}
\qquad
\\
\ll
{
N^{\gd(1-2\gs)}
\over
K^{1+\gd} q^{2}
}
+
N^{(1/2-\gs)(\frac 67 2\gd + \frac 5{21})}
,
\eeann
as $N\to\infty$.
\end{thm}

\pf
This follows from Theorem 1.15 in \cite{BourgainKontorovichSarnak2009}.
\epf




\section{Setup of the Exponential Sum and Major/Minor Arcs}\label{secSetup}

\subsection{The Exponential Sum $S_{N}$}
\

Let $\G<\SL(2,\Z)$ be finitely-generated, free, contain no parabolics, and have limit 
set with Hausdorff dimension $\gd>1/2$.
Fix two primitive vectors $v_{0},w_{0}\in\Z^{2}\setminus\{0\}$.
Recall the shape of the exponential sum function from \eqref{SNapprox}, $S_{N}(\gt)$
is approximately
$$
\sum_{\g_{1}\in\G\atop\|\g_{1}\|<N^{1/2}}
\sum_{\g_{2}\in\G\atop\|\g_{2}\|<N^{1/2}}
e
(
\<v_{0}\g_{1}\g_{2},w_{0}\>
\gt
)
,
$$
and this is used to develop certain bilinear forms estimates.

Assuming $\G$ has no parabolics, the vectors $v_{0}\g_{1}$ are all unique. 
But of course the products $\g_{1}\g_{2},$ though capable of reaching elements in 
$\G$ of norm $N$, can also back-track (e.g. if $\g_{2}=\g_{1}^{-1}$). 
%
%
To prevent this, use that $\G$ is free and restrict the range of $\g_{1}$ and $\g_{2}
$ further by writing each as a word in the generators of $\G$, and controlling the 
concluding letter of $\g_{1}$ and beginning letter of $\g_{2}$. Then it is not possible 
to back-track, and each element arising as the product $\g=\g_{1}\g_{2}$ is unique. 

Unfortunately, the above tweak is problematic for another reason. Namely, one 
would like to perform certain estimates involving spectral theory, in particular, 
appealing to the spectral gap. Conditions such as $\|\g\|<N$ can be encoded 
spectrally, whereas restrictions on letters appearing in representations of $\g$ as a 
word in the generators cannot.

So we add another element $\g_{3}$ of small norm, say $N^{\gs}$ with 
\be\label{gsBnd1}
\gs<1/4
.
\ee
(See \eqref{gsBnd2} for the place where this is used.)
This has the effect of ruining unique representations of $\g=\g_{1}\g_{2}\g_{3}$, but 
not by too much, while still allowing for estimates using spectral theory.
To distinguish their roles, we will call $\xi=\g_{1}$, $\vp=\g_{2}$, and $\g=\g_{3}$. 
The exact definition follows.

Fix a subset $\Xi=\Xi_{N^{1/2}}\subset \G$ consisting of  elements $\xi\in\G$ with $\|
\xi\|<N^{1/2}$, which when written (uniquely, since $\G$ is free) as a reduced word 
in the generators of $\G$,  all end with the same letter. By the pigeonhole 
principle, one can choose $\Xi$ so that the number of elements in $\Xi_{N^{1/2}}$ is 
$\gg N^{\gd}$ (the implied constant depending only on $\G$ and the number of 
generators of $\G$).

Also fix a subset $\Pi=\Pi_{N^{1/2-\gs}}\subset \G$ consisting of  elements $\vp\in\G
$ with $\|\vp\|<N^{1/2-\gs}$, which are written as a reduced word in the generators 
of $\G$, all starting with the same letter. (If this letter happens to be the inverse of 
the ending letter in $\Xi$, add a different fixed letter to the beginning of every element in $\Pi
$.) Again the number of elements in $\Pi_{N^{1/2-\gs}}$ is $\gg N^{\gd(1-2\gs)}$.

For large fixed $N$ and $\gt\in[0,1]$, define the exponential sum function by
\be\label{SNdef}
S_{N}(\gt):
=
\sum_{
\xi\in\Xi
} 
\
\sum_{\varpi\in\Pi} 
\
\sum_{
\g\in\G
\atop
\|\g\|<
N^{\gs}
} 
e( 
\<v_{0}\cdot
\g \xi\vp
,w_{0}\> 
\gt)
.
\ee
For $n\in\Z$, let 
\beann
R_{N}(n)&:=&\widehat{ S_{N}}(n)= \int_{0}^{1}S_{N}(\gt)e(-n\gt)d\gt
\\
&=&
\sum_{
\xi\in\Xi
} 
\
\sum_{\varpi\in\Pi} 
\
\sum_{
\g\in\G
\atop
\|\g\|<
N^{\gs}
} 
\bo_{\{
\<v_{0}\cdot
\g\xi\vp
,w_{0}\> 
=n\}}
\eeann
be
 the representation function.

\subsection{Major/Minor Arcs Decomposition}
\

Our  decomposition into major and minor arcs is made as follows. 
By Dirichlet's theorem, for every irrational number  $\gt\in[0,1]$ (and hence for 
almost every, with respect to Lebesgue measure), there is a $q<N^{1/2}$ and $
(a,q)=1$ such that $\gt=\frac aq+\gb$, with
$$
\left|
\gt-\frac aq
\right|
=
\left|
\gb
\right|<{1\over qN^{1/2}}
.
$$
Define the major arcs as the set of $\gt=\frac aq+\gb$ with $q<Q_{0}$ and $|\gb|<
{K_{0}/ N}$, where we set
$$
Q_{0}=N^{\ga_{0}},
\qquad
\text{and}
\qquad
K_{0}=N^{\gk_{0}}
.
$$
These are taken as large as possible while still controlling the main term, cf. Theorem \ref{thm:major}. 

For technical reasons of harmonic analysis, the main term is not simply the integral 
of the exponential sums over the major arcs. We mollify the sharp cutoff by 
introducing  certain weights. To this end,
let $\psi$ be  the triangle function 
\be\label{psiIs}
\psi(x)
:=
\threecase{0}{if $|x|\ge1$,}
{1-x}{if $0<x<1$,}
{1+x}{if $-1<x<0$.}
\ee
%
Then its Fourier transform is
$$
\hat\psi(y)
=
\left(
{\sin(\pi y)
\over \pi y}
\right)^{2}
;
$$
in particular it is positive.%
\footnote{
While of greatest importance to us is the positivity of the Fourier transform, we 
would also like $\psi$ to approximate the indicator function of the interval $[-1,1]$. 
Being positive makes $\hat\psi$ the square of some function, and so $\psi$ is a 
convolution of a function with itself. The convolution of the indicator function with itself is a 
triangle function, hence our choice of $\psi$.
}

Let
 $\Psi_{N,K_{0}}(\gb)$ be defined by
\be\label{PsiIs}
\Psi_{N,K_{0}}(\gb)
:=
\sum_{m\in\Z}
\psi
\bigg(
(\gb+m)N/K_{0}
\bigg)
.
\ee
In particular $\Psi$ is well-defined on the circle $[0,1]$ and is a triangle function about the interval $[-K_{0}/
N,K_{0}/N]$. 

Let the mollified major arcs ``indicator'' function be
\be\label{fMdef}
\fM(\gt):=
\sum_{1\le q<Q_{0}}
\sum_{(a,q)=1}
\Psi_{N,K_{0}}
\left(
\gt-\frac aq
\right)
.
\ee
Then define the main term by
\be\label{cMdef}
\cM_{N}(n)
:=
\int_{0}^{1}
\fM(\gt)
S_{N}(\gt)
e(-n\gt)
d\gt
.
\ee
The minor arcs function is then given by
$$
\fm(\gt)
:=
1-
\fM(\gt)
,
$$
and the error term is
$$
\cE_{N}(n)
:=
\int_{0}^{1}
\fm(\gt)
S_{N}(\gt)
e(-n\gt)
d\gt
.
$$

Recall that we intend  to bound the $\cL^{2}$ norm of $\cE_{N}$, that is, control the 
integral 
\vskip-.3in
\be\label{minL2}
\int_{0}^{1}
|
\fm(\gt)
S_{N}(\gt)
|^{2}
d\gt
.
\ee
For parameters $Q<N^{1/2}$ and $K<N^{1/2}$, we decompose  the circle into 
dyadic pieces of the form
$$
W_{Q,K}:=\left\{\gt=\frac aq+\gb:(a,q)=1,q\sim Q,|\gb|\sim {K\over N}\right\}
,
$$
where $x\sim X$ means $\foh X\le x <X$.
Then a bound for \eqref{minL2} will follow from controlling 
$$
\int_{W_{Q,K}}
|
S_{N}(\gt)
|^{2}
d\gt
,
$$
with either $Q>Q_{0}$ or $K>K_{0}$, or both. This is the task undertaken in Sections \ref
{secMin1}, \ref{secMin2}, and \ref{secMin3}.



\section{The Major Arcs}\label{secMajor}

This section is devoted to the proof of

\begin{thm}\label{thm:major}
There is a set $\fE(N)\subset[-N,N]$ of size $|\fE(N)|\ll N^{1-\vep_{0}}$ such that the 
following holds.
For $|n|<N$ and $n\notin\fE(N)$, the main term 
is
$$
\cM_{N}(n)
=
\twocase
{}
{\gg {1\over \log\log(10+|n|)}N^{2\gd-1}}
{if $n\in\cS(\mod q)$ for all $q\ge1$}
{0}
{otherwise,}
$$
provided that
$$
K_{0}=N^{\gk_{0}} \qquad\qquad\text{and}\qquad\qquad
Q_{0}=N^{\ga_{0}},
$$
with
\be\label{MajGk}
\gk_{0}< \frac32\gs(\gd-\foh),
\ee
and
\be\label{MajGaGk}
21\ga_{0}+{13}\gk_{0}<(2\gd - \frac5{3})\gs
.
\ee
\end{thm}

That is, we need to augment the exceptional set a bit
to allow the main term estimate to fail for a small number of $n$.

\subsection{Breaking into Modular and Archimedean Components}
\

Recall that the main term is
$$
\cM_{N}(n) = \int_0^{1}{\frak M}(\gt)S_{N}(\gt)e(-n\gt)d\gt,
$$
where the exponential sum is
$$
S_{N}(\gt)
=
\sum_{
\xi\in\Xi}
\sum_{
 \varpi\in\Pi} 
\
\sum_{
\g\in\G
\atop
\|\g\|<
N^{\gs}
} 
e( 
\<v_{0}\cdot
\g 
,w_{0}\,^{t}\vp\,^{t}\xi\> 
\gt)
,
$$
and the major arc weights $\fM(\gt)$ are given in terms of the triangle function $\psi
$ by:
$$
\fM(\gt)=
\sum_{1\le q<Q_{0}}
\sum_{(a,q)=1}
\Psi_{N,K_{0}}
\left(
\gt-\frac aq
\right)
.
$$

For fixed $\vp$ and $\xi$, write 
$$
w=w_{0}{}^{t}\vp\, {}^{t}\xi
.
$$
Let $\G_{1}(q)$ denote the group $\G_{1}(q):=\{\g\in\G:v_{0}\g\equiv v_{0}(q)\}$, 
which is clearly a ``congruence'' subgroup of $\G$ of level $q$.
Note that any $\g\in\G$ can we written as $\g=\g_{1}\g_{2}$, where $\g_{1}\in\G_{1}
(q)$, and $\g_{2}$ is a representative chosen from the quotient group $\G_{1}(q)\bk
\G$. Note further that 
$$
\<v_{0}\cdot
\g_{1}\g_{2}
,w\> 
\equiv 
\<v_{0}\cdot
\g_{2}
,w\> 
(\mod q)
.
$$ 
Therefore
\bea\label{MqOut}
&&
\sum_{
\g\in\G
\atop
\|\g\|<
N^{\gs}
} 
e\left( 
\<v_{0}\cdot
\g 
,w\> 
\left(\frac aq+\gb\right)
\right)
\\
\nonumber
&&
\qquad
=
\sum_{\g_{2}\in\G_{1}(q)\bk\G}
e\left( 
\<v_{0}\cdot
\g_{2} 
,w\> 
\frac aq
\right)
\sum_{
\g_{1}\in\G_{1}(q)
\atop
\|\g_{1}\g_{2}\|<
N^{\gs}
} 
e\left( 
\<v_{0}\cdot
\g_{1}\g_{2} 
,w\> 
\gb
\right)
\eea
Assume for simplicity that $\G$ has spectral gap $(\gT,\fB)$ with  $\gT=5/6$ and $\fB=1$. 
(The general case of $\fB>1$ is handled 
similarly.)
\begin{lem}
For $|\gb|<K_{0}/N$,
$$
\sum_{
\g_{1}\in\G_{1}(q)
\atop
\|\g_{1}\g_{2}\|<
N^{\gs}
} 
e\left( 
\<v_{0}\cdot
\g_{1}\g_{2} 
,w\> 
\gb
\right)
=
{1\over [\G:\G_{1}(q)]}
\sum_{
\g\in\G
\atop
\|\g\|<
N^{\gs}
} 
e\left( 
\<v_{0}\cdot
\g
,w\> 
\gb
\right)
+
O\left(K_{0}^{6/7}N^{\gs(\frac 67 2\gd + \frac 5{21})}\right)
.
$$
\end{lem}
\pf
This follows easily from Theorem \ref{thm:PassToQ}. 
\epf

Putting everything together, the main term is
\beann
\cM_{N}(n)
&=&
\sum_{\xi,\vp}\sum_{q<Q_{0}}\sum_{(a,q)=1}
\sum_{\g_{2}\in\G_{1}(q)\bk\G}
{
e\left( 
(
\<v_{0}\cdot
\g_{2} 
,w\> -n)
\frac aq
\right)
\over [\G:\G_{1}(q)]}
\\
&&
\times
\sum_{\g\in\G\atop\|\g\|<N^{\gs}}
\int_{0}^{1}
\Psi_{N,K_{0}}\left(\gb\right)
e\bigg((\<v_{0}\g,w
\>-n)\gb\bigg)
d\gb
\\
&&
\qquad
+
O\left(N^{2\gd(1-\gs)}Q_{0}^{3+\vep}{K_{0}\over N}K_{0}^{6/7}N^{\gs(\frac 67 2\gd + \frac 5{21})}\right)
,
\eeann
since $[\G:\G_{1}(q)]\ll q^{1+\vep}$.
Define the Ramanujan sum
$$
c_{q}(x):=\sum_{(a,q)=1}e(ax/q).
$$
Let the singular series be
$$
\fS_{N,\xi,\vp}(n)
:=
\sum_{q<Q_{0}}{1\over [\G:\G_{1}(q)]} \sum_{\g_{2}\in\G_{1}(q)\bk\G}c_{q}(\<v_{0}
\g_{2},w\>-n)
,
$$
and the singular integral be
$$
\tau_{N,\xi,\vp}(n)
:=
\sum_{\g\in\G\atop\|\g\|<N^{\gs}}\int_{0}^{1}\Psi_{N,K_{0}}(\gb)
e\bigg((\<v_{0}\g,w
\>-n)\gb\bigg)
d\gb
,
$$
so that
\be\label{cMBreak}
\cM_{N}(n) = \sum_{\xi,\vp}\fS_{N,\xi,\vp}(n)\tau_{N,\xi,\vp}(n)
+
O\left(N^{2\gd(1-\gs)}Q_{0}^{3+\vep}{K_{0}\over N}K_{0}^{6/7}N^{\gs(\frac 67 2\gd + \frac 5{21})}\right)
.
\ee
We have divided the main term into modular and archimedean components.

\subsection{The Archimedean Component}

\begin{lem}
For an integer $x$,
$$
\int_{0}^{1}\Psi_{N,K_{0}}(\gb)e(x\gb)d\gb
\ge
{2K_{0}\over 5N}\cdot \bo_{\{|x|<{N\over 2K_{0}}\}}
.
$$
\end{lem}
\pf
Inserting \eqref{PsiIs} gives
\beann
\int_{0}^{1}\Psi_{N,K_{0}}(\gb)e(x\gb)d\gb
&=&
\int_{\R}\psi\left(\gb {N\over K_{0}}\right)e(x\gb)d\gb
\\
&=&
{ K_{0} \over N}
\hat\psi\left(x{K_{0}\over N}\right)
.
\eeann
Using $\hat\psi(y)=\left({\sin(\pi y)\over \pi y}\right)^{2}$, one checks elementarily 
that $\hat\psi(y)>0.4$ for $|y|<1/2$.
\epf
 
 Applied to $\tau_{N,\xi,\vp}(n)$, the above gives
 $$
 \tau_{N,\xi,\vp}(n)
\ge
{2K_{0}\over 5N}
\sum_{\g\in\G\atop\|\g\|<N^{\gs}}
\bo_{\left\{
|\<v_{0}\g,w
\>-n|<{N\over 2K_{0}}
\right\}}
 $$
Applying Theorem \ref{thm:LowBnd} gives
 \bea
 \nonumber
 \tau_{N,\xi,\vp}(n)
&\gg&
{2K_{0}\over 5N}
{N^{2\gd\gs}\over K_{0}}
+
O
\left(
{K_{0}\over N}
N^{(2\gd+3)\gs/4}(\log N)^{1/4}
\right)
\\
\label{tauBnd}
&\gg&
{N^{2\gd\gs}\over N}
+
O
\left(
{K_{0}\over N}
N^{(2\gd+3)\gs/4}(\log N)^{1/4}
\right)
,
\eea
as long as $w=w_{0}{}^{t}\vp\, {}^{t}\xi$ satisfies $|w|\asymp N^{1-\gs}$. But if 
$|w|\ll N^{1-\gs-\vep}$, then $|\<v_{0}\g,w\>|\ll N^{1-\vep}$, and these values of $n$ 
may be discarded into the exceptional set $\fE(N)$.

With $K_{0}=N^{\gk_{0}}$, the bound \eqref{tauBnd} is significant as long as 
\be\label{K0is}
\gk_{0}
<
\frac32 \gs(\gd-\foh)
.
\ee

\subsection{Modular Component}
\

Assume that $\G(q)\bk\G\cong \SL(2,q)$ for all $q$; minor changes are need to accommodate the more general case. 
With $\xi$ and $\vp$ fixed and $w=w_{0}{}^{t}\vp\, {}^{t}\xi$, we evaluate the 
singular series:
\beann
\fS_{N,\xi,\vp}(n)=
\sum_{q<Q_{0}}{1\over \SL(2,q) 
} \sum_{\g\in\SL(2,q)
}c_{q}(\<v_{0}\g,w\>-n)
.
\eeann
Extend the  $q$ sum to infinity with a negligible error. 
The factors are  multiplicative, and the main contribution comes from the primes, 
not prime powers. Hence  we estimate just the prime contribution, which is
\beann
\prod_{p}
\left(
1
+
{1\over |\SL(2,p)|
} \sum_{\g\in\SL(2,p)
}c_{p}(\<v_{0}\g,w\>-n)
\right)
.
\eeann
By changing representatives, we may assume without loss of generality that $v_{0}
=w=(0,1)$, so that $\<v_{0}\g,w\>=d_{\g}$, the lower right entry.
The Ramanujan sum $c_{p}(x)$ on primes is either $p-1$ if $x\equiv 0(p)$ or $-1$ 
othewise. There are two cases.
\begin{enumerate}
\item Case $n\equiv 0(p)$: Suppose 
$
\g
=
\bigl( \begin{smallmatrix}
a&b\\ c&d
\end{smallmatrix} \bigr)
$ 
has $d\equiv n\equiv 0$. Then $b$ is invertible ($p-1$ choices) and uniquely 
determines $c$; $a$ is free ($p$ choices). The number of such $\g$ is thus $p^{2}-
p$.

The number of $\g$ with $d\neq n(p)$ is $p^{3}-p-(p^{2}-p)=p^{3}-p^{2}$ because 
$|\SL(2,p)|=p(p^{2}-1)$.
Hence
\beann
{1\over |\SL(2,p)|}
\sum_{\g\in\SL(2,p)} 
c_{p}
(
d
-n
)
&=&
{1\over p^{3}-p}
\bigg[
(p-1)(p^{2}-p)
+
(-1)
(p^{3}-p^{2})
\bigg]
\\
&=&
-
{
1
\over p+1}
.
\eeann

\item Case $n\neq 0(p)$: Suppose $\g$ has $d\equiv n\neq0(p)$. Then $b$ and $c
$ are free ($p^{2}$ choices) and $a$ uniquely determined
by $a=(1+bc)d^{-1}$.

The number of $\g$ with $d\neq n(p)$ is $p^{3}-p-(p^{2})=p^{3}-p^{2}-p$.
Hence
\beann
{1\over |\SL(2,p)|}
\sum_{\g\in\SL(2,p)} 
c_{p}
(
d
-n
)
&=&
{1\over p^{3}-p}
\bigg[
(p-1)(p^{2})
+
(-1)
(p^{3}-p^{2}-p)
\bigg]
\\
&=&
{
1
\over
p^{2}-1
}
.
\eeann

\end{enumerate}

To leading order, the singular series is thus
$$
\fS_{N,\xi,\vp}(n)
\gg
\prod_{p\nmid n}
\left(
1+
{1\over p^{2}-1}
\right)
\prod_{p\mid n}
\left(
1
-
{1\over p+1}
\right)
\gg {1\over \log\log n}
,
$$
as desired.

\subsection{Conclusion}
\

Returning to \eqref{cMBreak}, we have shown that for $n\notin\fE(N)$,
$$
\cM_{N}(n)
\gg 
{1\over \log\log n} N^{2\gd-1}
+
O\left(N^{2\gd(1-\gs)}Q_{0}^{3+\vep}{K_{0}\over N}K_{0}^{6/7}N^{\gs(\frac 67 2\gd + \frac 5{21})}\right)
.
$$
With $Q_{0}=N^{\ga_{0}}$,
this is conclusive if
\be\label{ga0gk0}
3\ga_{0}+\frac{13}7\gk_{0}<(2\gd - \frac5{3})\gs/7
.
\ee
This completes the proof of Theorem \ref{thm:major}.

\

\

\



\section{Minor Arcs I: $L^{\infty}$ norm of $S_{N}(\gt)$}\label{secMin1}

We establish in this section an $L^{\infty}$ bound for $S_{N}\bigg|_{W_{Q,K}}$. 
By itself, this bound is useful only if both $K$ and $Q$ are quite small, but it will be combined later with other estimates.

\begin{thm}\label{thm:min1}
Write $\gt=\frac aq +\gb$ with $q<N^{1/2}$, $|\gb|<{1\over qN^{1/2}}$, and $|\gb|
\sim {K\over N}$.
Then
\be
\label{Sbnd1b}
|S_{N}(\gt)|
\ll 
N^{(3\gd+1)/2}
\bigg(
{
1
\over
K^{(1+\gd)/2} q
}
+
N^{-\frac{1}{84} (6 \delta -5) (1-2 \sigma )}
\bigg)
.
\ee
\end{thm}

\subsection{Rewriting $S_{N}(\gt)$}
\

Write
\bea
\nonumber
S_{N}(\gt)
&=&
\sum_{\xi\in \Xi}
\sum_{\vp\in\Pi,\g\in\G\atop \|\g\|<N^{\gs}}
e(\<v_{0}\g\vp,w_{0}\,^{t}\xi\>\gt)
\\
\label{Sga}
&=&
\sum_{x}
\sum_{y}
\mu(x)
\nu(y)
e(\<x,y\>\gt)
,
\eea
where
$$
\mu(x)
=
\sum_{\xi\in\Xi}\bo(x=w_{0}\,^{t}\xi),
$$
and 
$$
\nu(y)
=
\sum_{\vp\in\Pi,\g\in\G\atop\|\g\|<N^{\gs}}\bo(y=v_{0}\g\vp).
$$

Note that  $x$ and $y$ are primitive vectors in $\Z^{2}$.
The assumption that $\G$ contains no parabolics implies that the values $w_{0}\,^
{t}\xi$ are unique, and hence 
\be\label{mmuBnd}
\mu\le1
.
\ee
Note also that
$$
\supp\mu,\ \supp\nu\subset B_{N^{1/2}}.
$$
On the other hand, $\nu$ does not have a unique decomposition. 
For $\g$ fixed, the value of $v_{0}\g\vp$ is unique, and there are $N^{\gd\gs}$ 
choices for $\g$. Hence we have, crudely,
\be\label{nnuBnd}
\nu
\ll
N^{\gd\gs}
.
\ee
The number of elements captured by each measure is 
\be\label{muAsss}
\sum_{x}\mu(x)\asymp \sum_{y}\nu(y)\asymp N^{\gd}.
\ee

\subsection{Bounding $S_{N}(\gt)$}
\

As the support of $\nu$ is in a ball in $\Z^{2}$ of radius $N^{1/2}$, we can break $
\nu$ into $64$ pieces $\nu=\sum_{\ga}\nu_{\ga}$, where the support of each piece 
$\nu_{\ga}$ is in a square of side length 
$$
{1\over 4}N^{1/2}
.
$$ 
Each piece obviously retains the bound $\nu_{\ga}\ll N^{\gd\gs}$, and we have the 
bound
$$
|S_{N}(\gt)|\le \sum_{\ga}|S_{\ga}(\gt)|,
$$
where
$$
S_{\ga}(\gt):=\sum_{x}\sum_{y}\mu(x)\nu_{\ga}(y)e(\<x,y\>\gt).
$$

Let $\Upsilon:\R^{2}\to\R$ be is a smooth, non-negative function which is at least $1$ in the unit 
square $[-1,1]^{2}$, and $
\supp\hat\Upsilon\in B_{\frac1{10}}$, say.

Apply Cauchy-Schwarz, insert $\Upsilon(x/N^{1/2})$ (since $\mu$ has support in a ball of radius $N^{1/2}$), and extend the 
$x$ sum to all of $\Z^{2}$ (effectively replacing the thin group $\G$ by the full modular group $\SL(2,\Z)$):
\beann
|
S_{\ga}(\gt)
|
&\ll&
\left(
\sum_{x\in\Z^{2}}
|
\mu(x)
|
^{2}
\right)^{1/2}
\left(
\sum_{x\in\Z^{2}}
\left|
\sum_{y\in\Z^{2}}
\nu_{\ga}(y)
e(\<x,y\>\gt)
\right|^{2}
\Upsilon\left({x\over N^{1/2}}\right)
\right)^{1/2}
\\
&\ll&
N^{\gd/2}
\left(
\sum_{x\in\Z^{2}}
\left|
\sum_{y\in\Z^{2}}
\nu_{\ga}(y)
e(\<x,y\>\gt)
\right|^{2}
\Upsilon\left({x\over N^{1/2}}\right)
\right)^{1/2}
,
\eeann
where we used   \eqref{mmuBnd} and \eqref{muAsss}.

Open the square, reverse  orders, and apply Poisson  summation:
\beann
\hskip-.6in
|
S_{\ga}(\gt)
|
&\ll&
N^{\gd/2}
\left(
\sum_{y\in\Z^{2}}
\sum_{y'\in\Z^{2}}
\nu_{\ga}(y)
\nu_{\ga}(y')
\sum_{x\in\Z^{2}}
e(\<x,y-y'\>\gt)
\Upsilon\left({x\over N^{1/2}}\right)
\right)^{1/2}
\\
&\ll&
N^{\gd/2}
\left(
\sum_{y\in\Z^{2}}
\sum_{y'\in\Z^{2}}
\nu_{\ga}(y)
\nu_{\ga}(y')
\sum_{k\in\Z^{2}}
\int_{x\in\R^{2}}
e(\<x,y-y'\>\gt)
\Upsilon\left({x\over N^{1/2}}\right)
e(-\<x,k\>)
dx
\right)^{1/2}
\\
&\ll&
N^{\gd/2}
N^{1/2}
\left(
\sum_{y\in\Z^{2}}
\sum_{y'\in\Z^{2}}
\nu_{\ga}(y)
\nu_{\ga}(y')
\sum_{k\in\Z^{2}}
\hat\Upsilon
\left(
N^{1/2}(\gt(y-y')-k))
\right)
\right)^{1/2}
.
\eeann
Since  $\supp\hat\Upsilon\in B_{\frac1{10}}$, there is at most one contribution in the $k$ sum, which is of size $\ll |\hat \Upsilon(0)|\ll1
$, and only occurs if 
$$
\|
\gt(y-y'))
\|
\le {1\over 10 N^{1/2}}.\footnote{
Here $\|\cdot\|$ is the distance to the 
nearest lattice point in $\Z^{2}$.}
$$
As $\gt=\frac aq+\gb$, we have
$$
\left|
\frac aq (y-y')
\right|
\le
|
\gt (y-y')
|
+
|\gb (y-y')|
.
$$
Note that since $\nu_{\ga}$ has support in a rectangle whose side lengths are $\le 
\frac 14 N^{1/2}$, we can bound 
$$
|y-y'|\le \frac{\sqrt2}4 N^{1/2}
.
$$
Since
$$
|\gb|\le {1\over qN^{1/2}},
$$
we have
$$
|\gb(y-y')|
\le
{1\over qN^{1/2}}
\frac{\sqrt 2}4N^{1/2}
\le
\frac{\sqrt 2}{4q}
.
$$
Using 
$
\|
\gt (y-y') 
\|
\le 
{1\over 10 N^{1/2}}
$
and $q<N^{1/2}$
gives 
$$
\|
\frac aq (y-y')
\|
\le
{1\over 10 N^{1/2}}
+
{ \sqrt 2 \over 4 q}
,
$$
or
$$
\|
a (y-y')
\|
\le
{1\over 10 }
{q\over N^{1/2}}
+
\frac{ \sqrt 2}4
<
0.5
,
$$
and hence 
$$
y\equiv y'(\mod q)
.
$$
Moreover,
$$
\| \gt(y-y')\|
=
{
|\gb(y-y')|
\le
{1\over 10 N^{1/2}}
.
}
$$
Putting everything together 
gives
\be\label{eq:sqrt}
|S_{N}(\gt)|
\ll
N^{\gd/2}
N^{1/2}
\left(
\sum_{y\in\Z^{2}} \nu(y)
\sum_{y' \in\Z^{2} \atop y\equiv y'(q), |y-y'|\le {1\over 10|\gb|N^{1/2}} }
\nu(y')
\right)^{1/2}
.
\ee

Let $|\gb|=\frac K{N}$, and assume that $K\gg1$ (the opposite case is handled 
similarly). We will analyze the  innermost sum as follows. 

Increase $\nu_{\ga}$ to all of $\nu$.
Fix $y$ and recall the definition of $\nu$. We wish to bound
\beann
\hskip-.7in
\sum_{y' \in\Z^{2} \atop y\equiv y'(q), |y-y'|\le {1\over 10|\gb|N^{1/2}} }
\nu(y')
&=&
\sum_{\vp\in\Pi}
\sum_{\g\in\G\atop\|\g\|<N^{\gs}}
\bo\bigg\{v_{0}\g\vp\equiv y(q), |v_{0}\g\vp-y|\le {1\over 10|\gb|N^{1/2}} \bigg\}
\\
&\le
&
\sum_{\g\in\G\atop\|\g\|<N^{\gs}}
\sum_{\vp\in\G\atop\|\vp\|<N^{1/2-\gs}}
\bo\bigg\{(c,d)\vp\equiv y(q), |(c,d)\vp-y|\le {1\over 10|\gb|N^{1/2}} \bigg\}
,
\eeann
where we have written $(c,d)=v_{0}\g$ and relaxed the condition $\vp\in\Pi$ (which 
constrains the starting letter of $\vp$ in addition to its norm) to just $\vp\in\G$ with $
\|\vp\|<N^{1/2-\gs}$.

Continuing to hold $y$ fixed, we also fix $(c,d)$, that is, fix $\g$. By throwing away 
small $n$'s into an exceptional set $\fE(N)$, we may restrict to those $y$ for which 
$|y|\asymp N^{1/2}$ and those $\g$ for which $|(c,d)|\asymp N^{\gs}$.
Now we are in position to apply Theorem \ref{thm:UpBnd} to the inner sum over $
\vp$. 
$$
\sum_{\vp\in\G\atop\|\vp\|<N^{1/2-\gs}}
\bo\bigg\{(c,d)\vp\equiv y(q), |(c,d)\vp-y|\le {N^{1/2}\over 10K} \bigg\}
\ll
{
N^{\gd-2\gd\gs}
\over
K^{1+\gd} q^{2}
}
+
N^{(1/2-\gs)(\frac67 2\gd+\frac5{21})}
.
$$

The sum over $\g$ contributes $N^{2\gd\gs}$ and the sum over $y$ contributes 
$N^{\gd}$.  
Inserting everything into \eqref{eq:sqrt} gives
\beann
S_{N}(\gt)
&\ll& 
N^{\gd/2}
N^{1/2}
\Bigg(
N^{\gd}
N^{2\gd\gs}
\bigg(
{
N^{\gd-2\gd\gs}
\over
K^{1+\gd} q^{2}
}
+
N^{(1-2\gs)(\frac67 \gd+\frac5{42})}
\bigg)
\Bigg)^{1/2}
\\
&\ll& 
N^{\gd/2}
N^{1/2}
\Bigg(
{
N^{\gd}
\over
K^{(1+\gd)/2} q
}
+
N^{\gd/2}
N^{\gd\gs}
N^{(1/2-\gs)(\frac67 \gd+\frac5{42})}
\Bigg)
\\
&\ll& 
{
N^{(3\gd+1)/2}
\over
K^{(1+\gd)/2} q
}
+
N^{\gd+1/2+\gd\gs+(1/2-\gs)(\frac67 \gd+\frac5{42})}
.
\eeann
This proves \eqref{Sbnd1b}.

\

\

\

\




\section{Minor Arcs II: Average $|S_{N}(\gt)|$ over $P_{\gb}$}\label{secMin2}

Now we fix $\gb$ and average $S_{N}(\gt)$.

\begin{thm}\label{thm:min2}
Recall that $\gt=\frac aq+\gb$, with $\foh Q\le q< Q<N^{1/2}$ and $|\gb|<{1\over 
qN^{1/2}}$.
Fix $\gb$ with $|\gb|<{2\over Q N^{1/2}}$
, and  let
$$
P_{Q,\gb}:=\bigg\{
\gt=\frac aq+\gb:
q\sim Q, (a,q)=1
\bigg\}
,
$$
so $|P|\asymp Q^{2}$.
Then 
\be\label{Sbnd3}
\sum_{\gt\in P_{Q,\gb}}
|S_{N}(\gt)|
\ll_{\vep}
N^{\gd+1+\vep}
Q
\bigg(
Q^{-1/2}
+
N^{-\gs}
+
N^{-\gs/2-1/2}
Q
\bigg)
.
\ee
\end{thm}

\

Write
\beann
S_{N}(\gt)
&=&
\sum_{
\xi\in\Xi
} 
\
\sum_{\varpi\in\Pi} 
\
\sum_{
\g\in\G
\atop
\|\g\|<
N^{\gs}
} 
e( 
\<v_{0}\cdot
\g \xi\vp
,w_{0}\> 
\gt)
\\
&=&
\sum_{
\xi\in\Xi
,\varpi\in\Pi} 
\
\sum_{
\g\in\G
\atop
\|\g\|<
N^{\gs}
} 
e(\
\<v_{0}
\g \
, w_{0}\,^{t}(\xi \vp) \> \
\gt)
\\
&=&
\sum_{x\in B_{N^{1-\gs}}}
\sum_{y\in B_{N^{\gs}}}
\mu(x)
\nu(y)
e(\<x,y\>\gt)
,
\eeann
where 
  $\mu$ and $\nu$ are 
now
   measures with $\supp\mu\subset B_{N^{1-\gs}}$ and $\supp
\nu\subset B_{N^{\gs}}$ defined by:
$$
\mu(x)
=
\sum_{\xi\in\Xi,\vp\in\Pi}\bo(x=w_{0}(\xi\vp)^{*})
,
$$
$$
\nu(y)
=
\sum_{\g\in\G\atop \|\g\|<N^{\gs}}\bo(y=v_{0}\g).
$$
Since products of the form $\xi\vp$ are unique, we have
$$
\mu,\
\nu
\le
1
.
$$

Write $|S_{N}(\gt)|=\gz(\gt)S(\gt)$, where $|\gz(\gt)|=1$. Then for any $\gW\subset
[0,1]$,
$$
\int_{\gW}|S_{N}(\gt)|d\gt
=
\int_{\gW}\gz(\gt)S_{N}(\gt)d\gt
=
\sum_{x}
\mu(x)
\sum_{y}
\int_{\gW}
\gz(\gt)
\nu(y)
e(\<x,y\>\gt)
d\gt
.
$$
Recall that $\Upsilon\in C^{\infty}(\R^{2})$ is a smooth, non-negative function which is at 
least $1$ in the unit square $[-1,1]^{2}$ with $\supp\hat\Upsilon\in B_{\frac1{10}}$.
Apply Cauchy-Schwarz, and insert $\Upsilon$ to retain the condition that $\supp\mu
\subset B_{N^{1-\gs}}$.
\bea
\nonumber
\int_{\gW}
|S_{N}(\gt)|
d\gt
&\ll&
\left(
\sum_{x}
\mu(x)^{2}
\right)^{1/2}
\left(
\sum_{x}
\left|
\sum_{y}
\int_{\gW}
\nu(y)
\gz(\gt)
e(\<x,y\>\gt)
d\gt
\right|^{2}
\Upsilon\left(\frac x{N^{1-\gs}}\right)
\right)^{1/2}
\\
\nonumber
&\ll&
N^{\gd(1-\gs)}
\Bigg(
\sum_{y}
\sum_{y'}
\int_{\gt\in\gW}
\int_{\gt'\in\gW}
\gz(\gt)
\overline{
\gz(\gt')}
\nu(y)
\nu(y')
\\
&&
\label{MinUseTwice}
\hskip1in
\times
\sum_{x}
e(\<x,y\gt-y'\gt'\>)
\Upsilon\left(\frac x{N^{1-\gs}}\right)
d\gt d\gt'
\Bigg)^{1/2}
\eea
In the above we used the bound 
$$
\sum_{x}
\mu(x)^{2}
\le
\sum_{x}
\mu(x)
\ll
N^{2\gd(1-\gs)}
.
$$

Write \eqref{MinUseTwice} with $\gW=P$ (and the integral as a sum).
Apply Poisson summation in the $x$ sum and use $\supp\hat\Upsilon\in B_{1/10}$, 
together with $|\gz|\le1$:
\be\label{SgtBndII}
\sum_{\gt\in P}|S_{N}(\gt)|
\ll
N^{\gd(1-\gs)}
N^{1-\gs}
\Bigg(
\sum_{y}
\sum_{y'}
\nu(y)\nu(y')
\sum_{\gt\in P}
\sum_{\gt'\in P}
\bo_{\{\|y\gt-y'\gt'\|<{1\over 10 N^{1-\gs}}\}}
\Bigg)^{1/2}
.
\ee

As $\nu\le1$ and supported on primitive vectors in  $B_{N^{\gs}}$, our task is then 
to count the number, say $A$, of points in the parentheses.

The set $A$ contains
\begin{enumerate}
\item
 lattice points $y=(y_{1},y_{2})\in B_{N^{\gs}}$,  and $y'=(y_{1}',y_{2}')\in B_{N^
{\gs}}$;
 and 
 \item
 points on the circle $\gt=\frac aq+\gb$ and $\gt'=\frac{a'}{q'}+\gb$ (same $\gb$ -- 
this is the key!).
 \end{enumerate}
which satisfy:
\begin{enumerate}
\item 
$|y|,|y'|<N^{\gs}$, primitive vectors,
\item 
 $y,y'\neq(0,0)$; and
\item
$
\|y_{1}\gt-y'_{1}\gt'\|<{1\over 10 N^{1-\gs}}
$,
 and
$
\|y_{2}\gt-y'_{2}\gt'\|<{1\over 10 N^{1-\gs}}
$,
where $\|\cdot\|$ is the distance to nearest integer.
\end{enumerate}

We note first that 
\beann
\left\|
\left(
y_{2}'
y_{1}
-
y'_{1}
y_{2}
\right)
\frac aq
\right\|
&=&
\left\|
y_{2}'
\left(
y_{1}\frac aq
-
y'_{1}{a'\over q'}
\right)
-
y'_{1}
\left(
y_{2}\frac aq
-
y'_{2}{a'\over q'}
\right)
\right\|
\\
&\le&
\left\|
y_{2}'
\left(
y_{1}\frac aq
-
y'_{1}{a'\over q'}
\right)
\right\|
+
\left\|
y'_{1}
\left(
y_{2}\frac aq
-
y'_{2}{a'\over q'}
\right)
\right\|
\eeann
and that,
since $|\gb|<{2\over QN^{1/2}}$,
\beann
\left\|
y_{1}\frac aq
-
y_{1}'{a'\over q'}
\right\|
&\le&
\left\|
y_{1}\gt
-
y_{1}'\gt'
\right\|
+
|
\gb
(y_{1}-y_{1}')
|
\\
&\le&
{1
\over
10 N^{1-\gs}}
+
{2\over Q N^{1/2}}
\cdot
2
N^{\gs}
.
\eeann
Therefore
$$
\left\|
y'_{2}
\left(
y_{1}\frac aq
-
y_{1}'{a'\over q'}
\right)
\right\|
\le
{1
\over
10 N}
+
{2\over Q N^{1/2}}
\cdot
2
N^{2\gs}
,
$$
and hence
$$
\left\|
\left(
y_{2}'
y_{1}
-
y'_{1}
y_{2}
\right)
\frac aq
\right\|
\le
2
\left(
{1
\over
10 N}
+
{2\over Q N^{1/2}}
\cdot
2
N^{2\gs}
\right)
,
$$
or
$$
\left\|
\left(
y_{2}'
y_{1}
-
y'_{1}
y_{2}
\right)
a
\right\|
\le
2
\left(
{q
\over
10 N}
+
{2q\over Q N^{1/2}}
\cdot
2
N^{2\gs}
\right)
\le
{1
\over
5 N^{1/2}}
+
8{1\over  N^{1/2}}
N^{2\gs}
.
$$
Choose $\gs$ so that $N^{\gs}<\frac14 N^{1/4}$, say. (this is where the condition $
\gs\approx \frac14$ arises, cf. \eqref{gsBnd1}). That is,
\be\label{gsBnd2}
{
\gs<{1/4}-{\log4\over \log N}.
}
\ee

Then the right hand side is $<1$, and so
$$
{
y_{2}'
y_{1}
-
y'_{1}
y_{2}
\equiv 
0
(q)
.
}
$$
In the same way, we deduce that
$$
y_{2}'
y_{1}
-
y'_{1}
y_{2}
\equiv 
0
(q')
,
$$
and hence
$$
{
y_{2}'
y_{1}
-
y'_{1}
y_{2}
\equiv 
0
(\tilde q)
,
}
$$
where $\foh Q\le \tilde q\le Q^{2}$ is the least common multiple of $q$ and $q'$.

The rest of the analysis breaks down into three regions: 
Either 
\begin{enumerate}
\item[(i)]
$
y_{1}y_{2}'
-
y_{2}y_{1}'
\neq
0
$
(but is $\equiv 0(\tilde q)$); or
\item[(ii)]
$
y_{1}y_{2}'
-
y_{2}y_{1}'
=0
$
but $y_{1}y_{2}y_{1}'y_{2}'\neq 0$; or
\item[(iii)]
$
y_{1}y_{2}'
-
y_{2}y_{1}'
=0
$
and 
$
y_{1}y_{2}
y_{1}'y_{2}'
=
0
.
$
\end{enumerate}
%

We handle these separately.

\subsection{Region (i)}
\
 
\begin{prop}\label{Reg1}
The contribution to $A$ from Region (i) is
$$
\ll N^{(1+\gd)2\gs
}Q
$$
\end{prop}

The proof is as follows. Write
$$
\tilde q
\mid
\left(
y_{2}'
y_{1}
-
y'_{1}
y_{2}
\right)
,
$$
and
$$
y_{1}y_{2}'
-
y_{2}y_{1}'
\neq
0
,
$$
so in particular $\tilde q\le 2 N^{2\gs}$. Recall also that $\tilde q<Q^{2}$, and hence
$$
\tilde q\le \min (Q^{2},2N^{2\gs})\le Q\sqrt 2 N^{\gs}
.
$$

\begin{lem}
$$
\|
y_{1}\frac aq
-
y'_{1}{a'\over q'}
\|
=0.
$$
\end{lem}
\pf
Assume not, then 
$
\|
y_{1}\frac aq
-
y'_{1}{a'\over q'}
\|
$
is at least $1/\tilde q$. But then (using $|\gb|<{2\over QN^{1/2}}$),
\beann
{1\over Q\sqrt 2 N^{\gs}}
&\le&
{1\over \tilde q}
\le
\left\|
y_{1}\frac aq
-
y'_{1}{a'\over q'}
\right\|
\le
\|
y_{1}\gt
-
y'_{1}\gt'
\|
+
|
\gb(
y_{1}
-
y'_{1})
|
\\
&\le&
{1\over 10 N^{1-\gs}}
+
{2\over Q N^{1/2}}
2N^{\gs}
,
\eeann
or (using $Q<N^{1/2}$ and $N^{2\gs}<{1\over 16} N^{1/2}$),
$$
{1\over \sqrt 2 }
\le
{Q N^{2\gs}\over 10 N}
+
{2\over  N^{1/2}}
2N^{2\gs}
\le
{Q {1\over 16}N^{1/2}\over 10 N}
+
{2\over  N^{1/2}}
2{1\over 16}N^{1/2}
\le
{1\over 160}+{1\over 4},
$$
which is obviously a contradiction. 
\epf

Now we have that 
$$
y_{1}\frac aq \equiv y'_{1}{a'\over q'}(\mod 1)
.
$$
The same argument of course applies to $y_{2}, \ y'_{2}$, that is,
$$
y_{2}\frac aq \equiv y'_{2}{a'\over q'}(\mod 1)
$$

Some more notation: Let $d=(q,q')$ and write $q=dq_{1}$, $q'=dq'_{1}$, with $(q_
{1},q'_{1})=1$. Recall that $(a,q)=1$.
Hence
\be\label{dqq}
y_{1} a{q'_{1}} \equiv y'_{1}{a' q_{1}}(\mod dq_{1}q'_{1})
.
\ee
\

Looking mod $q_{1}$ gives
$$
y_{1} a{q'_{1}} \equiv 0(\mod q_{1})
,
$$
which forces
$$
{
q_{1}\mid y_{1},
}
$$
since $(q_{1},aq'_{1})=1$. The same argument applies to show that
$$
q_{1}\mid y_{2},
\qquad
q'_{1}\mid y'_{1},
\qquad
\text{ and }
\qquad
q'_{1}\mid y'_{2}.
$$
But since $y$ is a primitive vector, $(y_{1},y_{2})=1$, and hence 
$$
q_{1}=1
\qquad
\text{ and }
\qquad
d=q
.
$$
By the same token, $q'_{1}=1$ and $d=q'$, so in fact (!)
$$
{
q=q'.
}
$$

Then \eqref{dqq} and its companion become
$$
y_{1}a\equiv y'_{1}a'(q),
\qquad
\text{ and }
\qquad
y_{2}a\equiv y'_{2}a'(q).
$$

We count the contribution to $A$ as follows. There are $\ll Q$ choices for $q$, then 
$\ll q^{2}$ choices for $a, a'$. There are $\ll N^{2\gs\gd}$ choices for primitive pairs 
$(y_{1},y_{2})$. Then $y'_{1}$ and $y'_{2}$ are determined mod $q$, and hence 
there are $\ll N^{2\gs}q^{-2}$ choices for them, crudely (we are not using any spectral theory here!). Altogether, the contribution  is
\be\label{A1}
\ll
\sum_{q\sim Q}
q^{2}
N^{2\gs\gd}
{N^{2\gs}\over q^{2}}
\ll
N^{2\gs(1+\gd)}Q
.
\ee 

This proves Proposition \ref{Reg1}.

\subsection{\bf Region (ii)}
\

\begin{prop}\label{Reg2}
The contribution to $A$ from Region (ii) is
$$
\ll_{\vep}
N^{2\gs\gd+\vep}
Q^{2}
+
N^{2\gs\gd+\gs-1+\vep}
Q^{4}
,
$$
for any $\vep>0$.
\end{prop}

Recall that in this region, 
$$
y_{1}y_{2}'
-
y_{2}y_{1}'
=
0
,
$$
and so
$$
\tilde q
\mid
\left(
y_{2}'
y_{1}
-
y'_{1}
y_{2}
\right)
$$
is vacuous. We also have in this region that $y_{1},$ $y_{2},$ $y'_{1},$ $y'_{2},$ 
are all non-zero. Moreover, the vectors $y$ and $y'$ are primitive. By unique 
factorization, 
$
y_{1}y_{2}'
=
y_{2}y_{1}'
$ 
forces
$$
{
y_{1}=\pm y'_{1}
,
\qquad
\qquad
y_{2}=\pm y'_{2}
.
}
$$
Hence there are ${\ll N^{2\gs\gd}}$ choices for $y,y'$.

Let $q_{1}=(y_{1},q)$ and $q'_{1}=(y'_{1},q')=(y_{1},q')$. As $q_{1},q'_{1}\mid y_
{1}$, there are ${\ll_{\vep} N^{\vep}}$ choices for $q_{1}$ and $q'_{1}$. Assume 
without loss of generality that $q_{1}\le q'_{1}$. Fix $a'$ and $q'$, for which there are 
\be\label{QQbnd1}
{\ll Q \cdot {Q
\over q_{1}'}}
\ee 
choices.

Write $y_{1}=q_{1}z_{1}$ and $q=q_{1}q_{2}$. Then
$
\|
y_{1}\frac aq
-
y'_{1}{a'\over q'}
+
\gb
(y_{1}-y'_{1})
\|
<
{1\over 10 N^{1-\gs}}
$
becomes
$$
\|
z_{1}\frac a{q_{2}}
-
\psi
\|
<
{1\over 10 N^{1-\gs}}
,
$$
where $\psi=y'_{1}{a'\over q'}
+
\gb
(y_{1}-y'_{1})
$ 
is already  fixed.

The grid in the unit interval  of possible values of 
$z_{1}\dfrac a{q_{2}}$
 as $a$ and $q_{2}$ vary has mesh of size at least 
$$
\dfrac{4q_{1}^{2}}{Q^{2}} 
.
$$
  Hence the set of values 
of
$z_{1}\dfrac a{q_{2}}$
 satisfying the above proximity to $\psi$ is 
\be\label{QQbnd2}
\ll
{
\frac
{Q^{2}}
{4q_{1}^{2}}
{1\over N^{1-\gs}} +1
}
\ee

Let 
$z_{1}\dfrac a{q_{2}}\equiv \tilde \psi(\mod 1)$ for some fixed grid point $\tilde\psi$. 
Since 
$$
(q_{2},a,z_{1})=1
,
$$ 
this determines $q_{2}$ uniquely. Then $a$ is determined $(\mod q_{2})$, so has 
\be\label{QQbnd3}
{q\over q_{2}}={q_{1}}
\ee 
possible values.

Combining \eqref{QQbnd1}, \eqref{QQbnd2}, and \eqref{QQbnd3}, we have that the contribution to $A$ from Region (ii) is at 
most:
\beann
&&
\sum_{y}\nu(y)
\sum_{y'=\pm y}
\sum_{q_{1}|y_{1}, q_{1}'|y_{1}\atop q_{1}\le q_{1}'\le \min (Q,N^{\gs})}
Q\cdot {Q\over q'_{1}}
\left(
1+
{Q^{2}\over 4 q_{1}^{2}}
{1\over N^{1-\gs}}
\right)
q_{1}
%
\\
&\ll_{\vep}&
N^{2\gs\gd+\vep}
Q^{2}
+
N^{2\gs\gd+\gs-1+\vep}
Q^{4}
,
\\
\eeann
for any $\vep>0$, as claimed.

\subsection{Region (iii):}
\

\begin{prop}\label{Reg3}
The contribution to $A$ from Region (iii) is
$$
\ll
N^{\gs}
Q^{2}
.
$$
\end{prop}

Recall that in this region, 
$$
y_{1}y_{2}'
=
y_{2}y_{1}'
,
$$
and  $y_{1}=0$, say. By primitivity, $y_{2}=\pm1$, and by the above, $y_{1}'=0$ 
and again $y_{2}'=\pm1$.

The analysis is now the same as in Region (ii) except $q_{1}=q'_{1}=1$, so there 
are no $\vep$'s. The contribution is
\beann
&\ll &
Q\cdot {Q}
\left(
1+
{Q^{2}\over 4 }
{1\over N^{1-\gs}}
\right)
\ll
Q^{2}
+
Q^{4}N^{\gs-1}
\\
&\ll &
Q^{2}
+
Q^{2}N^{\gs}
\ll
Q^{2}N^{\gs}
.
\eeann

\subsection{Conclusion}
\

Combining \eqref{SgtBndII} with Propositions \ref{Reg1}, \ref{Reg2}, and \ref{Reg3}, 
gives
\beann
\sum_{\gt\in P}|S(\gt)|
&\ll&
N^{\gd(1-\gs)}
N^{1-\gs}\
A
^{1/2}
\\
&\ll_{\vep}&
N^{\gd(1-\gs)}
N^{1-\gs}\
\left(
N^{2\gs(1+\gd)
}
Q
+
N^{2\gs\gd+\vep}
Q^{2}
+
N^{2\gs\gd+\gs-1+\vep}
Q^{4}
+
N^{\gs}Q^{2}
\right)
^{1/2}
\\
&\ll_{\vep}&
N^{\gd(1-\gs)}
N^{1-\gs}\
Q
N^{\gs+\gs\gd+\vep
}
\left(
Q^{-1/2}
+
N^{-\gs}
+
N^{-\gs/2-1/2}
Q
\right)
,
\eeann
as claimed.
This completes the proof of Theorem \ref{thm:min2}.

\

\

\

\



\section{Minor Arcs III: Average  $|S_{N}|^{2}$ over $W_{Q,K}$}\label{secMin3}

The main goal of this section is to prove

\begin{thm}
\label{thm:min3}
\be\label{2g}
\int_{W_{Q,K}}|S_{N}(\gt)|^{2}d\gt
\ll
\log N\
\bigg(
N^{\gd
+\gd\gs}
\left\|
S_{N}(\gt)\bigg|_{W_{Q,K}}
\right\|_{L^{\infty}}
+
N^{2\gd+1-\gs}
\bigg)
.
\ee
\end{thm}

We begin with a lemma.

\begin{lem}
Let
$$
\gW\subset[0,1]
$$
be a finite union of open intervals.
Then 
\be\label{2c}
\int_{\gW}
|S_{N}(\gt)|
d\gt
\ll
\max
\bigg(
N^{\gd
+\gd\gs}
\
,
\
N^{\gd+(1-\gs)/2}
|\gW|^{1/2} 
%
\bigg)
.
\ee
\end{lem}

\pf

Returning to \eqref{MinUseTwice}, apply Poisson summation in the $x$ sum and use $\supp\hat\Upsilon\in B_{1/10}$, 
together with $|\gz|\le1$:
$$
\hskip-.3in
\int_{\gW}|S_{N}(\gt)|d\gt
\ll
N^{\gd(1-\gs)}
N^{1-\gs}
\left(
\sum_{y}
\sum_{y'}
\nu(y)
\nu(y')
\
meas\bigg\{(\gt,\gt'):
{
\|\gt y_{1}-\gt' y_{1}'\|<{1\over N^{1-\gs}}
\atop
\|\gt y_{2}-\gt' y_{2}'\|<{1\over N^{1-\gs}}
}
\bigg\}
\right)^{1/2}
.
$$
Let $Y:=\mattwo {y_{1}}
{y_{2}}
{-y_{1}'}
{-y_{2}'}$ and consider two regions:    either $\det Y=0$ or not.

 If not, then $(\gt,\gt')\mapsto (\gt,\gt')Y=(\gt y_{1}-\gt' y_{1}',\gt y_{2}-\gt'y_{2}')$ is 
a map which is measure preserving $(\mod 1\times 1)$. 
 Hence the preimage has
 $$
meas\bigg\{(\gt,\gt'):
{
\|\gt y_{1}-\gt' y_{1}'\|<{1\over  N^{1-\gs}}
\atop
\|\gt y_{2}-\gt' y_{2}'\|<{1\over  N^{1-\gs}}
}
\bigg\}
\ll 
N^{2\gs-2}
, 
 $$
and there are $N^{4\gs\gd}$ choices for $y,y'$.

If on the other hand the determinant of $Y$ is zero, then by primitivity, $y=\pm y'$, so there are only $N^{2\gs
\gd}$ choices. Assume $y_{1}\neq0$. Fix $\gt'\in\gW$ (contributing at most $|\gW|
$); then $\gt$ satisfies $\|y_{1}\gt-\gt_{0}\|<{1\over N^{1-\gs}}$ for some fixed $\gt_
{0}$. Hence the contribution is at most
$$
N^{2\gs\gd}|\gW| N^{\gs-1}
.
$$

Combining the two regions gives
\beann
\int_{\gW}|S_{N}(\gt)|d\gt
&\ll&
N^{\gd(1-\gs)}
N^{1-\gs}
\left(
N^{4\gs\gd}
N^{2\gs-2}
+
N^{2\gs\gd}|\gW| N^{\gs-1}
\right)^{1/2}
\\
&\ll&
N^{\gd
+\gd\gs}
+
N^{\gd+(1-\gs)/2}
|\gW|^{1/2} 
,
\eeann
as claimed.
\epf



\pf[Proof of Theorem \ref{thm:min3}]
Let $W_{Q,K}=\gW$.
As $|S_{N}(\gt)|\ll N^{2\gd}$, we can take a dyadic subdivision $M\ll N^{2\gd}$ of  $
\ll\log N$ terms, 
and 
decompose 
$
\gW
$
into level sets 
$$
\gW=\bigsqcup_{\ga}\gW_{\ga}
,
$$ 
according to the size of $|S_{N}(\gt)|$.
So if $\gt\in\gW_{\ga}$, then  $\frac M2\le |S_{N}(\gt)|<M$ with $M\ll N^{2\gd}$.

On any such level set, we have
$$
\frac1{|\gW_{\ga}|} \int_{\gW_{\ga}}|S_{N}(\gt)|d\gt 
\asymp 
\sup_{\gt\in\gW_{\ga}} |S_{N}(\gt)|
,
$$
so
\beann
\int_{\gW}|S_{N}(\gt)|^{2}d\gt
&\ll&
\log N\
\sup_{\ga}\
\int_{\gW_{\ga}}|S_{N}(\gt)|^{2}d\gt
\\
&\ll&
\log N\
\sup_{\ga}\
\sup_{\gt\in\gW_{\ga}} |S_{N}(\gt)|
\int_{\gW_{\ga}}|S_{N}(\gt)|d\gt
\\
&\ll&
\log N\
\sup_{\ga}\
\sup_{\gt\in\gW_{\ga}} |S_{N}(\gt)|
\max
\bigg(
N^{\gd
+\gd\gs}
,
N^{\gd+(1-\gs)/2}
|\gW_{\ga}|^{1/2} 
\bigg)
\\
&\ll&
\log N\
\max
\bigg(
\sup_{\gt\in\gW} |S_{N}(\gt)|
N^{\gd
+\gd\gs}
,
\\
&&
\sup_{\ga}\
N^{\gd+(1-\gs)/2}
|\gW_{\ga}|^{-1/2} 
\int_{\gW_{\ga}}|S_{N}(\gt)|d\gt
\bigg)
\\
&\ll&
\log N\
\max
\bigg(
\sup_{\gt\in\gW} |S_{N}(\gt)|
N^{\gd
+\gd\gs}
\
,
\
N^{2\gd+(1-\gs)}
\bigg)
\\
&\ll&
\log N\
\bigg(
N^{\gd
+\gd\gs}
\left\|
S_{N}(\gt)\bigg|_{\gW}
\right\|_{L^{\infty}}
+
N^{2\gd+1-\gs}
\bigg)
,
\eeann
where we used \eqref{2c} twice (the second time under the assumption $N^{\gd+
\gd\gs}\ll N^{\gd+(1-\gs)/2}|\gW_{\ga}|^{1/2}$).
This completes the proof.
\epf

\

\

\

\



\section{Putting It All Together}\label{secFinish}

Now we combine the previous estimates to show that the minor arcs are small. We will 
write 
$$
Q=N^{\ga}\qquad\qquad\text{ and }\qquad\qquad K=N^{\gk},
$$ 
with $\ga,\gk\in[0,1/2]$.
For the vast majority of $\gt$, the weight $\fm(\gt)$ is identically one. We deal with 
these $\gt$'s
first, via the following three Lemmata.

\begin{lem}\label{S86}
As $N\to\infty$,
$$
\int_{W_{Q,K}}|S_{N}(\gt)|^{2}d\gt
\ll 
N^{4\gd-1-\eta}
,
$$
if
\be\label{S86b}
{
\ga+\frac{1+\gd}2 \gk > \frac32(1-\gd)+\gd\gs
,
}
\ee
\be\label{NewEq1}
\gs>2(1-\gd)
,
\ee
and
\be\label{S86a}
{
\gs < {132 \gd - 131\over 96 \delta -10}
}
.
\ee
\end{lem}
This Lemma is conclusive if either $K$ or $Q$ is large.
\pf
Putting  \eqref{Sbnd1b} into \eqref{2g} and ignoring log's gives:
$$
\int_{W_{Q,K}}|S_{N}(\gt)|^{2}d\gt
\ll
N^{\gd
+\gd\gs}
N^{(3\gd+1)/2}
\bigg(
{
1
\over
K^{(1+\gd)/2} Q
}
+
N^{-\frac1{84}(1-2\gs)(6\gd-5)}
\bigg)
+
N^{2\gd+1-\gs}
.
$$
For the last term of the right-hand side to be bounded by $N^{4\gd-1-\eta}$, we need
$$
\gs>2(1-\gd)
.
$$
For the middle term above to be bounded by $N^{4\gd-1-\eta}$, we need
$$
{
\gs < {132 \gd - 131\over 96 \delta -10}
}
.
$$
Lastly, the first term above is controlled if
$$
{
\ga+\frac{1+\gd}2 \gk > \frac32(1-\gd)+\gd\gs
.
}
$$
As $\gd$ is very near $1$, this essentially requires that $\ga+\gk>
\gs
$ in order to bound the minor arcs outright. 
\epf
If both $K$ and $Q$ are too small to apply the above, but not small enough to be in the major arcs, we try  the next lemma, which is conclusive if $K$ is almost as small as the major arcs. 

\begin{lem}\label{S82}
As $N\to\infty$,
$$
\int_{W_{Q,K}}|S_{N}(\gt)|^{2}d\gt
\ll 
N^{4\gd-1-\eta}
,
$$
if
\be\label{S82a}
\gk
>
\frac{1-\gd}\gd
,
\ee
and
\be\label{S82b}
{
1-\gd+
\gk+2\ga
<
\frac1{42}(6\gd-5)(1-2\gs)
}
.
\ee
\end{lem}
\pf
Using  \eqref{Sbnd1b} twice leads to the bound
\beann
\int_{W_{Q,K}}|S_{N}(\gt)|^{2}d\gt
&\ll&
Q^{2}
\frac KN
\bigg(
{
N^{3\gd+1}
\over
K^{1+\gd} Q^{2}
}
+
N^{-\frac1{42}(6\gd-5)(1-2\gs)}
\bigg)
\\
&=&
{
N^{3\gd}
\over
K^{\gd}
}
+
Q^{2} K
N^{3\gd}
N^{-\frac1{42}(6\gd-5)(1-2\gs)}
.
\eeann
In order for this to be sufficient, we need both
$$
\gk
>
\frac{1-\gd}\gd
,
$$
and
$$
{
1-\gd+
\gk+2\ga
<
\frac1{42}(6\gd-5)(1-2\gs)
}
,
$$
as claimed.
\epf

The remaining case is when $K$ and $Q$ are too small to apply Lemma \ref{S86}, and $K$ is too small for Lemma \ref{S82} to suffice, that is, when $Q$ is small, but not so small as to be in the major arcs.

\begin{lem}\label{S84}
As $N\to\infty$,
$$
\int_{W_{Q,K}}|S_{N}(\gt)|^{2}d\gt
\ll 
N^{4\gd-1-\eta}
,
$$
if the conditions \eqref{S84a} through \eqref{S84f} are satisfied.
\end{lem}
\pf
Now add  \eqref{Sbnd1b} to \eqref{Sbnd3}:
\beann
\int_{W_{Q,K}}|S_{N}(\gt)|^{2}d\gt
&\ll&
\sup |S_{N}(\gt)|
\frac KN
\sum_{P}|S_{N}(\gt)|
\\
&\ll&
N^{(3\gd+1)/2}
\bigg(
{
1
\over
K^{(1+\gd)/2} Q
}
+
N^{-{1\over 84}(6\gd-5)(1-2\gs)}
\bigg)
\frac KN\
\\
&&
\times
N^{\gd+1+\vep}
Q
\bigg(
Q^{-1/2}
+
N^{-\gs}
+
N^{-\gs/2-1/2}
Q
\bigg)
.
\eeann
For these to all be conclusive, that is, $\ll N^{4\gd-1-\eta}$, we need the following six conditions.
\be\label{S84a}
{
(1-\gd)\gk + 3(1-\gd)
<
\ga
}
\ee
and
\be\label{S84b}
{
(1-\gd)\gk/2+
3(1-\gd)/2
<
\gs
}
\ee
and
\be\label{S84c}
{
(1-\gd)\gk/2+
3(1-\gd)/2
+
\ga
<
(1+\gs)/2
}
\ee
and
\be\label{S84d}
{
3(1-\gd)/2
+\gk + \ga/2
<
\frac1{84}(6\gd-5)(1-2\gs)
}
\ee
and
\be\label{S84e}
{
3(1-\gd)/2
+\gk +\ga
<
\frac{1}{84}(6\gd-5)(1-2\gs)
+\gs
}
\ee
and
\be\label{S84f}
{
3(1-\gd)/2
+\gk+2\ga
<
\frac1{84}(6\gd-5)(1-2\gs)
+\gs/2+1/2
.
}
\ee
\epf

Lastly, 
dispose of those $\gt$ with non-trivial weights in the minor arcs function $\fm(\gt)$. Recall from \eqref{psiIs} the triangle function $\psi$ used to form $\fM(\gt)$ and $\fm(\gt)$. Note that in $[-1,1]$ the function $1-
\psi(x)$ is just $|x|$. These values of $\gt$ (which should be 
contained in the major arcs but receive some small weights here) are controlled as 
follows:
\beann
\int_{
\fm(\gt)\neq1}
|\fm(\gt)|^{2}
|S_{N}(\gt)|^{2}
d\gt
&\ll&
\sum_{q<Q}
\sum_{(a,q)=1}
\int_{|\gb|<K/N}
(\gb N/K)^{2}
{
N^{3\gd+1}
\over
(N\gb)^{1+\gd}Q^{2}
}
d\gb
\\
&\ll&
{
N^{3\gd}
\over
K^{\gd}
}
,
\eeann
using the first term in the bound
\eqref{Sbnd1b} ignoring the second term, which was handled 
in \eqref{S82b}.  This is 
estimated in the same way as in Lemma \ref{S82}.

\subsection{Conclusion}\label{S87}
\

We must now collect all of the inequalities needed above and try to make sense of them. 

We need to find values of $\gd$ and $\gs$ such that every pair $(\ga,\gk)$ either lies in the major arcs, that is, satisfies \eqref{MajGk} and \eqref{MajGaGk}, or lies in one of the three regions described in Lemmata \ref{S86}, \ref{S82}, and \ref{S84}.

The main roles in $\ga$ and $\gk$ are played by \eqref{MajGaGk} in the major arcs, and \eqref{S82a} and \eqref{S84a} in the minor arcs. Putting these together with the main condition \eqref{S86b} on $\gd$ and $\gs$ gives the system of inequalities:
\be\label{eq:Region}
\twocase{}
{\gs>2(1-\gd)}{}
{21\left((1-\gd)\left(\dfrac{1-\gd}{\gd}\right)+3(1-\gd)\right) + 13\left(\dfrac{1-\gd}{\gd}\right)<\left(2\gd-\frac53\right)\gs.}{}
\ee
The region in the $(\gd,\gs)$ plane for which \eqref{eq:Region} and \eqref{S86a} are satisfied is depicted in Figure \ref{Fig3}. Solving for the minimal value of $\gd$ in this region, one finds that this value is
The minimal value for $\gd$ is the largest root of the cubic polynomial:
\be\label{eq:gdPoly}
1020 - 8897 x - 5010 x^2 + 12888 x^3,
\ee
whose plot is shown in Fig. \ref{Fig4}. The approximate value is
\be\label{gdRIS}
\gd\approx
0.9999493550
   .
\ee

\begin{figure}
\begin{center}
\includegraphics[width=2in]{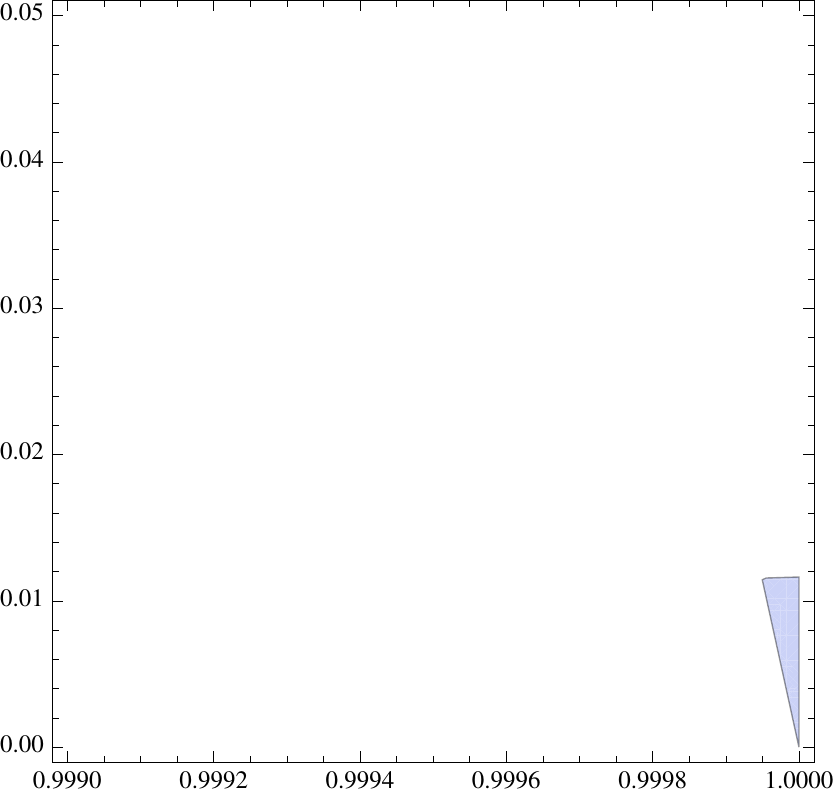}
\end{center}
\caption{Values in the $(\gd,\gs)$ plane satisfying \eqref{eq:Region} and \eqref{S86a}.}
\label{Fig3}
\end{figure}

\begin{figure}
\begin{center}
\includegraphics[width=2in]{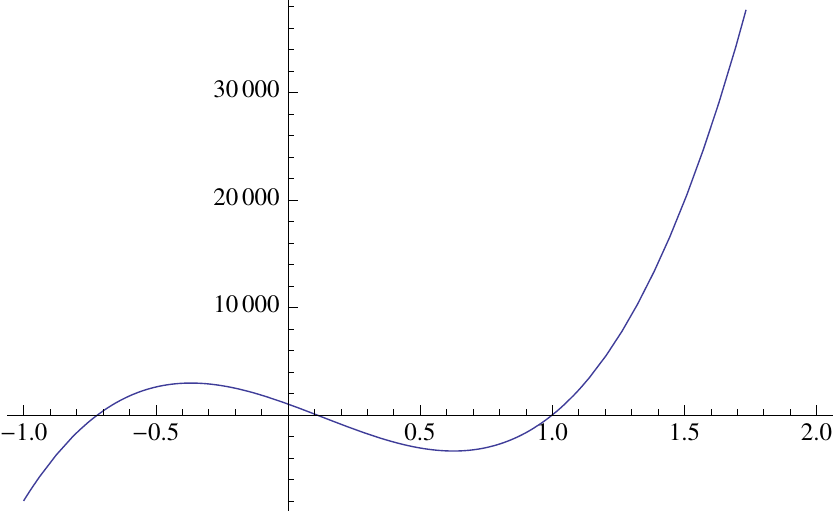}
\end{center}
\caption{The plot of the polynomial \eqref{eq:gdPoly}.}
\label{Fig4}
\end{figure}

The corresponding value of $\gs$ is the root nearest the origin of the polynomial
$$
4995 - 434163 x + 149452 x^2 + 700 x^3,
$$
which takes the approximate value
$$
\gs\approx
0.011550825843
.
$$

With these values of $\gd$ and $\gs$, Fig. \ref{Fig5} shows the overlapping regions, corresponding to the major arcs, and Lemmata \ref{S86}, \ref{S82}, and \ref{S84}.

The largest region corresponds to  Lemma \ref{S86}, where we see the negatively sloped line corresponding to the condition  \eqref{S86a}. The 
triangle to the right on the bottom corresponds to Lemma \ref{S82}, since it contains the 
vertical line showing  the inequality $\gk>{1-\gd\over \gd}$. The 
triangle on the left side 
corresponds to  Lemma \ref{S84}. On the magnified image, the vertical line is again Lemma \ref{S82} and the horizontal line is Lemma \ref{S84}. The triangle based at the origin corresponds to the major arcs, with the negatively sloped line being the condition \eqref{MajGaGk}. \\

\begin{figure}
\begin{center}
\includegraphics[width=2in]{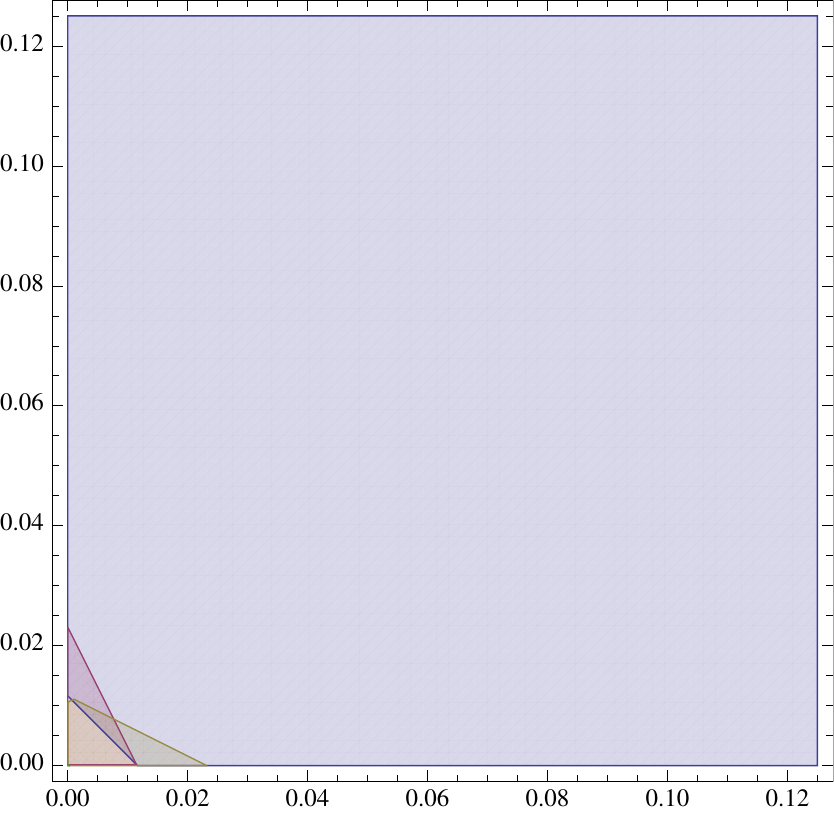}\qquad
\includegraphics[width=2in]{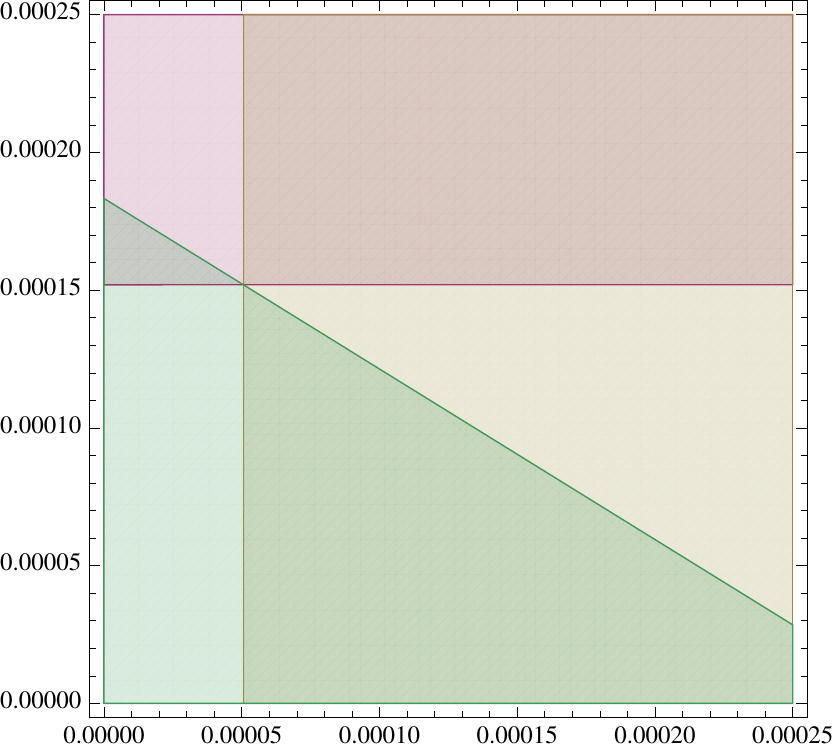}
\end{center}
\caption{Overlapping regions at different magnifications in the $(\gk,\ga)$-plane corresponding to the major arcs and Lemmata \ref{S86}, \ref{S82}, and \ref{S84}.}
\label{Fig5}
\end{figure}

We have proved:

\begin{thm}\label{thm:minor}
Assume that 
$$
\gd>0.9999493550
.
$$ 
Then there is a choice for $\gs$ such that as
$N\to\infty$, there is some $\eta>0$ with
$$
\sum_{|n| < N}|\cE_{N}(n)|^{2}
=
\int_{0
}^{1}
|\fm(\gt)|^{2}
|
S_{N}(\gt)|^{2}
d\gt
\ll 
N^{4\gd-1-\eta}
.
$$
\end{thm}

We conclude with the standard argument below.

\begin{thm}\label{thm:finish}
Theorems \ref{thm:major} and \ref{thm:minor} imply Theorem \ref{thm:main}.
\end{thm}
\begin{proof}
Let $\frak E(N)$ be the set of exceptions up to $N$. Let $\cZ$ denote those integers passing local obstructions. Then
\begin{align*}
|\frak E(N)|
&=&
{\displaystyle
\sum_{| n|<N\atop {n\in\cZ, |\cE_{N}(n)|>\cM_{N}(n)}}1
}&
\\
&\ll&
{\displaystyle
\sum_{| n|<N\atop |\cE_{N}(n)|\gg {1\over \log\log n}N^{2\gd-1}}1
}&
\\
&\ll&
{\displaystyle
\sum_{| n|<N}|\cE_{N}(n)|^{2} (\log\log N)^{2}N^{2-4\gd}
}&
\\
&\ll&
{\displaystyle
{
N^{1-\eta}
 (\log\log N)^{2}
 }
}.
&
\end{align*}
This completes the proof.
\end{proof}

\bibliographystyle{alpha}

\bibliography{../../AKbibliog}

\end{document}